\documentclass[a4paper,11pt,DIV=11,abstract=on]{scrartcl}
\usepackage[subsection]{algorithm}
\usepackage{algpseudocode,algorithm,algorithmicx}
\usepackage{amsfonts,mathrsfs}
\usepackage{acro}
\usepackage{amsmath}
\usepackage{dsfont}
\usepackage{tipa}
\usepackage{amssymb}
\usepackage{textcomp}
\usepackage{amsthm}
\usepackage{mathtools}
\usepackage{enumitem}
\usepackage{subcaption}
\usepackage{adjustbox}
\usepackage{here}
\usepackage{hyperref}
\usepackage{xcolor}
\usepackage{authblk}
\usepackage[capitalise]{cleveref}

\usepackage[english]{babel}
\usepackage[utf8]{inputenc}

\usepackage[T1]{fontenc}
\usepackage{lmodern}

\hypersetup{colorlinks=true,linkcolor=black,citecolor=black}

\newtheorem{definition}{Definition}[subsection]
\newtheorem{theorem}[definition]{Theorem}

\newtheorem{corollary}[definition]{Corollary}
\newtheorem{proposition}[definition]{Proposition}

\newtheorem{remark}[definition]{Remark}

\makeatletter 
\let\c@algorithm\c@definition
\makeatother

\newcommand{\Real}{\mathbb{R}}
\newcommand{\domain}{\Omega}

\newcommand{\N}{\mathbb{N}}

\newcommand{\R}{\mathbb{R}}

\newcommand{\dint}{\,\mathrm{d}} 
\DeclareMathOperator{\Div}{div} 
 
\newcommand{\grad}{\nabla}

\newcommand{\Diff}{\mathcal{D}}

\crefname{section}{Sec.}{Secs.}
\Crefname{section}{Section}{Sections}
\crefname{theorem}{Thm.}{Thms.}
\Crefname{theorem}{Theorem}{Theorems}
\crefname{proposition}{Prop.}{Props.}
\Crefname{proposition}{Proposition}{Propositions}
\crefname{cor}{Corr. \nameref}{Corrs. \nameref}
\Crefname{cor}{Corr. \nameref}{Corrs. \nameref}
\crefname{lem}{Lemma \nameref}{Lemmas. \nameref}
\Crefname{lem}{Lemma \nameref}{Lemmas \nameref}
\crefname{equation}{}{}
\Crefname{equation}{}{}
\crefname{item}{}{}
\Crefname{item}{}{}
\crefname{figure}{Fig.}{Figs.}
\Crefname{figure}{Figure}{Figures}

\title{Template-Based Image Reconstruction from Sparse Tomographic Data}

\author[1]{Lukas F. Lang}
\author[2]{Sebastian Neumayer}
\author[3]{Ozan \"{O}ktem}
\author[1]{Carola-Bibiane Sch\"{o}nlieb}
\affil[1]{\footnotesize Department of Applied Mathematics and Theoretical
Physics, University of Cambridge, Wilberforce Road, Cambridge CB3 0WA, United Kingdom}
\affil[2]{\footnotesize Department of Mathematics, Technische Universit\"{a}t Kaiserslautern, D-67663 Kaiserslautern, Germany}
\affil[3]{\footnotesize Department of Mathematics, KTH-Royal Institute of Technology, 100 44 Stockholm, Sweden}

\begin{document}

\date{}
\maketitle

\begin{abstract}
We propose a variational regularisation approach for the problem of template-based image reconstruction from indirect, noisy measurements as given, for instance, in X-ray computed tomography.
An image is reconstructed from such measurements by deforming a given template image.
The image registration is directly incorporated into the variational regularisation approach in the form of a partial differential equation that models the registration as either mass- or intensity-preserving transport from the template to the unknown reconstruction.
We provide theoretical results for the proposed variational regularisation for both cases.
In particular, we prove existence of a minimiser, stability with respect to the data, and convergence for vanishing noise when either of the abovementioned equations is imposed and more general distance functions are used.
Numerically, we solve the problem by extending existing Lagrangian methods and propose a multilevel approach that is applicable whenever a suitable downsampling procedure for the operator and the measured data can be provided.
Finally, we demonstrate the performance of our method for template-based image reconstruction from highly undersampled and noisy Radon transform data.
We compare results for mass- and intensity-preserving image registration, various regularisation functionals, and different distance functions.
Our results show that very reasonable reconstructions can be obtained when only few measurements are available and demonstrate that the use of a normalised cross correlation-based distance is advantageous when the image intensities between the template and the unknown image differ substantially.
\end{abstract}

\section{Introduction} \label{sec:intro}

In medical imaging, an image can typically not be observed directly but only through indirect and potentially noisy measurements, as it is the case, for example, in computed tomography (CT) \cite{Nat01}.
Due to the severe ill-posedness of the problem, reconstructing an image from measurements is rendered particularly challenging when only few or partial measurements are available.
This is, for instance, the case in limited-angle computed tomography \cite{Fri13,Nat01}, where limited-angle data is acquired in order to minimise exposure time of organisms to X-radiation.
Therefore, it can be beneficial to impose a priori information on the reconstruction, for instance, in the form of a template image.
However, typically neither its exact position nor its exact shape is known.

In \emph{image registration}, the goal is to find a reasonable deformation of a given template image so that it matches a given target image as closely as possible according to a predefined similarity measure, see \cite{Mod09, Mod03} for an introduction.
When the target image is unknown and only given through indirect measurements, it is referred to as \emph{indirect image registration} and has been explored only recently \cite{CheOkt18, GriCheOkt18, HinSzeWanSalJos12, OktCheDomRavBaj17}.
As a result, a deformation together with a transformed template can be computed from tomographic data.
The prescribed template acts as a prior for the reconstruction and, when chosen reasonably close in a deformation sense, gives outstanding reconstructions in situations where only few measurements are available and competing methods such as filtered backprojection \cite{Nat01} or total variation regularisation \cite{RudOshFat92} fail, see \cite[Sec.~10]{CheOkt18}.

In our setting, deformations are maps from the image domain $\Omega \subset \R^{n}$, $n \in \N$, to itself together with an action that specifies exactly how such a map deforms elements in the shape space, which in this work is the space $L^{2}(\Omega, \R)$ of greyscale images supported in the image domain.
Natural problems are to characterise admissible deformations and to compute these numerically in an efficient manner.

One possible approach is diffeomorphic image registration, where the set of admissible deformations is restricted to diffeomorphisms in order to preserve the topology of structures within an image \cite{You10}.
One can, for instance, consider the group of diffeomorphisms together with the composition as group operation.
Elements in this group act on greyscale images by means of the group action and thereby allow for a rich set of non-rigid deformations, as required in many applications.
For instance, the geometric group action transforms greyscale images in a way such that its intensity values are preserved, whereas the mass-preserving group action ensures that, when the image is regarded as a density, the integral over the density is preserved.

A computational challenge in using the above group formalism is that it lacks a natural vector space structure, which is typically desired for the numerical realisation of the scheme.
Hence, it is convenient to further restrict the set of admissible deformations.
One way to obtain diffeomorphic deformations is to perturb the identity map with a displacement vector field.
Provided that the vector field is reasonably small and sufficiently regular, the resulting map is invertible \cite[Prop. 8.6]{You10}.
For indirect image registration this idea was pursued in \cite{OktCheDomRavBaj17}.

The basic idea of the \emph{large deformation diffeomorphic mapping (LDDMM)} \cite{BegMilTroYou05, DupGreMil98, MilTroYou06, MilTroYou15, Tro98, TroYou15, You10} framework is to generate large deformations by considering flows of diffeomorphisms that arise as the solution of an ordinary differential equation (ODE), the so-called flow equation, with velocity fields that stem from a reproducing kernel Hilbert space.
In order to ensure that the flow equation admits a unique solution, one typically chooses this vector space so that it can be continuously embedded into $C^1(\Omega, \R^{n})$, allowing the application of existence and uniqueness results from Cauchy--Lipschitz theory for ODEs, see \cite[Chap.~1]{Cri07} for a brief introduction.
In \cite{CheOkt18}, the LDDMM framework is adapted for indirect image registration and the authors prove existence, stability, and convergence of solutions for their variational formulation.
Numerically, the problem is solved by gradient descent.

The variational problem associated with LDDMM is typically formulated as an ODE-constrained optimisation problem.
As the flow equation can be directly related to hyperbolic partial differential equations (PDE) via the method of characteristics \cite[Chap.~3.2]{Eva10}, the problem can equivalently be rephrased as a PDE-constrained optimisation problem \cite{HonJosSanStyNie12}.
The resulting PDE is determined by the chosen group action, see \cite[Sec.~6.1.1]{CheOkt18} for a brief discussion.
For instance, the geometric group action is associated with the transport (or advection) equation, while the mass-preserving group action is associated with the continuity equation.
It is important to highlight that the PDE constraint implements both the flow equation and the chosen diffeomorphic group action.

Such an optimal control approach was also pursued for motion estimation and image interpolation \cite{AndSchZul15, BorItoKun02, BorItoKun03, BurDirScho18, CheLor11, HarZacNie09, NieHarZac09}.
In the terminology of optimal control, the PDE represents the state equation, the velocity field the control, and the transformed image the resulting state.
We refer to the books \cite{BorSchu12, DeL15, Gun03, HinPinUlbUlb09} and to the article \cite{HerKun10} for a general introduction to PDE-constrained optimisation and suitable numerical methods.
Let us mention that other methods, such as geodesic shooting \cite{Ash07, MilTroYou06, SinHinJosFle13, ViaRisRueCot11}, exist and constitute particularly efficient numerical approaches.
In particular, this direction has recently been combined with machine learning methods \cite{YanKwiStyNie17}.

A particularly challenging scenario for diffeomorphic image registration occurs when the target image is not contained in the orbit of the template image under abovementioned group action of diffeomorphisms.
For instance, this could happen in the case of the geometric group action due to the appearance of new structures in the target image or due to a discrepancy between the image intensities of the template and the target image.
A possible solution is provided by the \emph{metamorphosis} framework \cite{MilYou01, RicYou15, TroYou05a, TroYou05}, which is an extension to LDDMM that allows for modulations of the image intensities along characteristics of the flow.
The image intensities change according to an additional flow equation with an unknown source.
See \cite[Chap.~13]{You10} for a general introduction and, for instance, \cite{HonJosSanStyNie12} for an application to magnetic resonance imaging.
Let us also mention \cite{NeuPerSte18a}, which adopts a discrete geodesic path model for the purpose of image reconstruction, and \cite{MaaRumSchoSim15}, in which the metamorphosis model is combined with optimal transport.

In \cite{GriCheOkt18}, the metamorphosis framework is adapted for indirect image registration.
The authors prove that their formulation constitutes a well-defined regularisation method by showing existence, stability, and convergence of solutions.
However, in the setting where only few measurements---e.g.\ a few directions in CT---are available, reconstruction of appearing or disappearing structures seems very challenging.

Therefore, in order to obtain robustness with respect to differences in the intensities between the transformed template and the sought target image, we follow a different approach.
We consider not only the standard \emph{sum-of-squared differences (SSD)} but also a distance that is based on the \emph{normalised cross correlation (NCC)} \cite[Chap.~7.2]{Mod09}, as it is invariant with respect to a scaling of the image intensities.

While image registration itself is already an ill-posed inverse problem that requires regularisation \cite{EngHanNeu96}, the indirect setting as described above is intrinsically more challenging.
It can be phrased as an inverse problem, where measurements (or observations) $g \in Y$ are related to an unknown quantity $f \in X$ via the operator equation
\begin{equation}
	K(f) = g + n^{\delta}.
\label{eq:operatorequation}
\end{equation}
Here, $K\colon X \to Y$ is a (not necessarily linear) operator that models the data acquisition, often by means of a physical process, $n^{\delta}$ are measurement errors such as noise, and $X$ and $Y$ are Banach spaces.
When $f$ constitutes an image and $g$ are tomographic measurements, solving \cref{eq:operatorequation} is often referred to as \emph{image reconstruction}.

We use a variational scheme \cite{SchGraGroHalLen09} to solve the inverse problem of indirect image registration, which can be formulated as a PDE-constrained optimisation problem \cite[Sec.~6.1.1]{CheOkt18}.
It is given by
\begin{equation}\label{eq:fun_to_min}
\begin{aligned}
	\min_{v \in V} & \ J_{\gamma, g}(v), \\
	\text{s.t.} & \ C(v),
\end{aligned}
\end{equation}
where $J_{\gamma, g}\colon V \to [0, + \infty]$ is the functional
\begin{equation}\label{eq:fun_to_min2}
	 v \mapsto D(K(f_{v}(T, \cdot)), g) + \gamma \Vert v \Vert_V^2.
\end{equation}
Here, $V$ is an admissible vector space with norm $\Vert \cdot \Vert_{V}$, $D\colon Y \times Y \to \R_{\ge 0}$ is a data fidelity term that quantifies the misfit of the solution against the measurements, and $\gamma > 0$ is a regularisation parameter.
Moreover, $f_{v}(T, \cdot)\colon \Omega \to \R$ denotes the evaluation at time $T > 0$ of the (weak) solution of $C(v)$, which is either the Cauchy problem
\begin{equation*}
	C(v) = \begin{cases}
		\frac{\partial}{\partial t} f(t, x) + v(t, x) \nabla_{x} f(t, x) = 0, & \text{for } (t, x) \in [0, T] \times \Omega, \\
		f(0, x) = f_{0}(x), & \text{for } x \in \Omega,
	\end{cases}
\end{equation*}
governed by the transport equation, or
\begin{equation*}
	C(v) = \begin{cases}
		\frac{\partial}{\partial t} f(t, x) + \mathrm{div}_{x}\bigl(v(t, x) f(t, x)\bigr) = 0, & \text{for } (t, x) \in [0, T] \times \Omega, \\
		f(0, x) = f_{0}(x), & \text{for } x \in \Omega,
	\end{cases}
\end{equation*}
involving the continuity equation.
Here, $f_{0} \in L^2(\Omega,\R)$ denotes an initial condition, which in our case is the template image.

The main goals of this article are the following.
First, to study variational and regularising properties of problem \cref{eq:fun_to_min}, and to develop efficient numerical methods for solving it.
Second, to investigate alternative choices of distance functions $D$, such as the abovementioned NCC-based distance.
Third, to demonstrate experimentally that excellent reconstructions can be computed from highly undersampled and noisy Radon transform data.

Our numerical approach is based on the Lagrangian methods developed in \cite{ManRut17}, called LagLDDMM.
In contrast to most existing approaches, which are mainly first-order methods (see \cite{ManRut17} for a brief classification and discussion), LagLDDMM uses a Gauss--Newton--Krylov method paired with Lagrangian solvers for the hyperbolic PDEs listed above.
The characteristics associated with these PDEs are computed with an explicit Runge--Kutta method.
One of the main advantages of this approach is that Lagrangian methods are unconditionally stable with regard to the admissible step size.
Furthermore, the approach limits numerical diffusion and, in order to evaluate the gradient or the Hessian required for optimisation, does not require the storage of multiple space-time vector fields or images at intermediate time instants.
The scheme can also be implemented matrix-free.

In comparison to abovementioned existing methods for indirect image registration, such as \cite{CheOkt18, GriCheOkt18, HinSzeWanSalJos12, OktCheDomRavBaj17}, our method is conceptually different in several ways.
The first difference concerns the discretisation.
While \cite{CheOkt18, GriCheOkt18, OktCheDomRavBaj17} are mainly based on small deformations and use reproducing kernel Hilbert spaces, our method relies on nonparametric registration.
The main advantages are that it directly allows for a multilevel approach and no kernel parameters need to be chosen.
Moreover, due to the flexibility of the underlying framework it is straightforward to extend our method to parametric registration.
Second, our approach relies on second-order methods for optimisation by using a Gauss--Newton method paired with line search, while the other methods mainly rely on gradient descent.
This allows for a fast decrease of the objective within only few iterations.
Third, our method allows to easily exchange the underlying PDE solver.
Essentially, any solver can be used as long as it can be differentiated efficiently.
The used explicit Runge--Kutta method has the advantage that it does not require the storage of multiple images or repeated interpolation of the template, which can potentially lead to a blurred solution.
Finally, let us mention that \cite{HinSzeWanSalJos12} is conceptually different since both a deformation and a template image are computed.
Our main focus, however, are applications where only very few and noisy measurements are available and the problem of estimating an additional template seems highly underdetermined in such situations.

\paragraph{Contributions}

The contributions of this article are as follows.
First, we provide the necessary theoretical background on (weak) solutions of the continuity and the transport equation, and recapitulate existence and uniqueness theory for characteristic curves for the associated ODE.
In contrast to the results derived in \cite{CheOkt18}, where the template image is assumed to be contained in the space $SBV(\Omega, \R) \cap L^{\infty}(\Omega, \R)$ of essentially bounded functions with special bounded variation, our results only require $L^{2}(\Omega, \R)$ regularity.
In addition, by using results from \cite{DipLio89}, we are able to consider the transport equation in the setting with $H^1$ regularity of vector fields in space as well as in time and with bounded divergence.
Moreover, we show the existence of a minimiser of the problem \cref{eq:fun_to_min}, stability with respect to the data, and convergence for vanishing noise.

Second, in order to solve the problem numerically, we follow a \emph{discretise-then-optimise} approach and extend the LagLDDMM framework \cite{ManRut17} to the indirect setting.
The library itself is an extension of FAIR \cite{Mod09} and, as a result, our implementation provides great flexibility regarding the selected PDE, and can easily be extended to other distances as well as to other regularisation functionals.
The source code of our MATLAB implementation is available online.\footnote{\url{https://doi.org/10.5281/zenodo.2598138}}

Finally, we present numerical results for the abovementioned distances and PDEs.
To the best of our knowledge, the results obtained for indirect image reconstruction based on the continuity equation are entirely novel.
Moreover, we propose to use the NCC-based distance instead of SSD whenever the image intensities of the template and the unknown target are far apart, and show its numerical feasibility.

\section{Theoretical results on the transport and continuity equation} \label{sec:background}
In this section, we review the necessary theoretical background, and state results on the existence and stability of weak solutions of the transport and the continuity equation. Compared to \cite{CheOkt18}, our results are stronger since we do not require space regularity of the template image.
\subsection{Continuity equation}
In what follows, we consider well-posedness of the continuity equation that arises in the LDDMM framework using the mass-preserving group action via the method of characteristics.
The regularity assumptions on $v$ are such that we can apply the theory from \cite{TroYou05a}.

Let $\Omega \subset \Real^n$ be a bounded, open, convex domain with Lipschitz boundary and let $T>0$.
In the following, we examine the continuity equation
\begin{equation}\label{eq:PDE}
	\begin{cases}
	\dfrac{\partial}{\partial t} f(t,x) + \Div_x \bigl(v(t,x) f(t,x) \bigr)  = 0 & \text{for $(t,x)\in [0,T]\times\Omega$,} \\[.5em] 
	f(0,x) = f_0(x) & \text{for $x\in\Omega$},
	\end{cases}
\end{equation}
with coefficients $v \in L^2([0,T],\mathcal V)$ and initial condition $f_0 \in L^2(\Omega, \Real )$, where $\mathcal V$ is a Banach space which is continuously embedded into $C^{1,\alpha}_0(\domain, \Real^n)$ for some $\alpha>0$.
Here $C^{1,\alpha}_0(\domain, \Real^n)$ denotes the closure of $C^\infty_c(\domain, \Real^n)$ under the $C^{1,\alpha}$ norm.
Note that such velocity fields can be continuously extended to the boundary.
Clearly, equation \cref{eq:PDE} has to be understood in a weak sense, i.e.\ a function $f \in C^0([0,T], L^2(\Omega, \Real))$ is said to be a weak solution of \cref{eq:PDE} if
\begin{equation}\label{eq:SimpleWeakFormulation}
	\int_{0}^{T}\int_{\Omega} f(t,x) \Bigl(v(t,x) \grad_x  \eta(t,x) + \frac{\partial}{\partial t} \eta(t,x) \Bigr) \dint x \dint t 
	+ \int_{\Omega} f_0(x) \eta(0,x) \dint x = 0
\end{equation}
holds for all $\eta \in C^\infty_c([0,T) \times \Omega)$.
The corresponding characteristic ODE is
\begin{align}\label{eq:ODE}
	&\begin{cases}\displaystyle{\frac{\dint}{\dint t}} X(t,s,x) = v\bigl( t, X(t,s,x) \bigr)   
	\quad \text{for } (t,s,x) \in [0,T]\times[0,T]\times\Omega, &
	\\[0.5em] 
	X(s,s,x)= x \quad \text{for } x\in \Omega.
	\end{cases}  
\end{align}
In this notation, the first argument of $X$ is the time dependence, the second the initial time, and the third the initial space coordinate.
The following theorem is a reformulation of \cite[Thms.~1 and 9]{TroYou05a} and characterises solutions of \cref{eq:ODE}.
\begin{theorem}\label{thm:DiffeoVelo}
	Let $v \in L^2([0, T], \mathcal V)$ and $s \in [0,T]$ be given.
	There exists a unique global solution $X(\cdot,s,\cdot) \in C^0([0,T],C^1(\overline{\Omega}, \R^n))$ such that $X(s,s, x) = x$ for all $x \in \domain$ and
	\begin{equation*}
		\frac{\dint}{\dint t} X(t,s,x) =
	v(t, X(t,s,x))
	\end{equation*}
	in weak sense (absolutely continuous solutions).
	The solution operator $X_v \colon L^2([0,T],\mathcal V) \to C^0([0,T] \times \overline \Omega, \R^n)$ assigning a flow $X_v$ to every velocity field $v$ is continuous with respect to to the weak topology in $L^2([0,T],\mathcal V)$.
\end{theorem}
Since $X(0,t,X(t,0,x)) = x$, we can directly conclude that $X(t,0,\cdot)$ is a diffeomorphism for every $t \in [0,T]$. Now, the diffeomorphism $X(0,t,x)$ can be used to characterise solutions of \cref{eq:PDE} as follows.

\begin{proposition}
	If $v \in L^2([0, T], \mathcal V)$, then the unique weak solution of \cref{eq:PDE}, as defined in \cref{eq:SimpleWeakFormulation}, is given by $f(t,x) = \det (\Diff_x X(0,t,x)) f_0(X(0,t,x))$, where $\Diff_{x} X$ denotes the Jacobian of $X$.
\end{proposition}
\begin{proof}
	The proof is divided in three steps.
	First, we want to show that $f$ satisfies the regularity conditions of weak solutions. For this purpose, the first step is to show $X(0,\cdot,\cdot) \in C^0([0,T],C^0(\overline{\Omega}, \R^n))$, i.e.\ that the flow is continuous in the initial values.
	Clearly, $X(0,t,\cdot)\in C^0(\overline{\Omega}, \R^n)$ for every $t \in [0,T]$.
	For an arbitrary sequence $t_i \to t$ we get
%	\begin{align*}
%	&\Vert X(0,t_n,\cdot)-X(0,t,\cdot)\Vert_{C^0(\overline{\domain})}\\
%	&\leq \Vert \Diff_x X(0,t_n,\cdot) \Vert_{C^0(\overline{\domain})} \Vert x-X(t_n,t,\cdot)\Vert_{C^0(\overline{\domain})}\\
%	&= \Vert (\Diff_x X(t_n,0,\cdot))^{-1}(X(0,t_n,\cdot)) \Vert_{C^0(\overline{\domain})} \Vert x-X(t_n,t,\cdot)\Vert_{C^0(\overline{\domain})}\\
%	&= \Vert (\Diff_x X(t_n,0,\cdot))^{-1} \Vert_{C^0(\overline{\domain})} \Vert x-X(t_n,t,\cdot)\Vert_{C^0(\overline{\domain})} \to 0,
%	\end{align*}
%	where the first factor is bounded since $X(\cdot,0,\cdot) \in C^0([0,T],C^1(\overline{\Omega}, \R^n))$ and matrix inversion is continuous.
	\begin{align*}
	\Vert X(0,t_i,\cdot)-X(0,t,\cdot)\Vert_{C^0(\overline{\domain})}
	\leq \Vert \Diff_x X(0,t_i,\cdot) \Vert_{C^0(\overline{\domain})} \Vert \text{Id} -X(t_i,t,\cdot)\Vert_{C^0(\overline{\domain})} \to 0,
	\end{align*}
	where the first factor is bounded due to \cite[Lemma~9]{TroYou05a}.
	Next, using the sequence $X_i(\cdot) = X(0,t_i,\cdot)$, it follows $f_0(X(0,\cdot,\cdot)) \in C^0([0,T], L^2(\Omega, \Real))$, where the continuity in time follows from \cite[Cor.~3]{NeuPerSte18}.
	Then, by differentiating $X(0,t,X(t,0,x)) = x$ and rearranging the terms we obtain
\begin{equation*}
	\det (\Diff_x X(0,\cdot,\cdot)) = \det (\Diff_x X(\cdot,0,\cdot))^{-1}(X(0,\cdot,\cdot)) \in C^0([0,T]\times \overline{\Omega}),
\end{equation*}
	since all involved expressions are continuous. Finally, we conclude $f \in C^0([0,T], L^2(\Omega, \Real))$, which follows from
	\begin{multline}
		\Vert f(t,\cdot) - f(t_i,\cdot)\Vert_{L^2(\Omega)}\\
		\leq\Vert \det (\Diff_x X(0,t,x)) - \det (\Diff_x X(0,t_i,x))\Vert_{C^0(\overline{\domain})} \Vert f_0(X(0,t,x)) \Vert_{L^2(\Omega)}\\
		+ \Vert \det (\Diff_x X(0,t_i,x)) \Vert_{C^0(\overline{\domain})} \Vert f_0(X(0,t,x)) - f_0(X(0,t_i,x)) \Vert_{L^2(\Omega)},
	\end{multline}
	since both summands converge to zero.
	
	The second step is to show that \cref{eq:SimpleWeakFormulation} is satisfied. Note that $X(\cdot,0,x)$ is differentiable in $t$ for a.e.\ $t \in [0,T]$, since it is absolutely continuous by definition.
	By inserting $f$ into \cref{eq:SimpleWeakFormulation} and using the transformation formula, we get
	\begin{multline}
	\int_{0}^{T}\int_{\Omega} f(t,x) \left( v(t,x) \grad_x \eta(t,x) 
	+ \frac{\partial \eta(t,x)}{\partial t} \right)\dint x \dint t + \int_{\Omega} f_0(x) \eta(0,x) \dint x\\
	= \int^{T}_0\int_{\Omega} \det \bigl( \Diff_x X(t,0,x)\bigr) f\bigl(t,X(t,0,x)\bigr)  
	\dfrac{d}{dt} \eta\bigl(t,X(t,0,x)\bigr) \dint x \dint t 
	+ \int_{\Omega} f_0(x) \eta(0,x) \dint x\\
	= \int^{T}_0\int_{\Omega} f_0(x)  
	\dfrac{d}{dt} \eta\bigl(t,X(t,0,x)\bigr) \dint x \dint t 
	+ \int_{\Omega} f_0(x) \eta(0,x) \dint x = 0.
	\end{multline}
	For the last equality we used that $\eta(t,X(t,0,x))$ is absolutely continuous.
	
	The last step is to prove uniqueness of weak solutions, i.e.\ that every solution has the given form.
	Let $f_1,f_2$ be two different solutions, then we can find a $t$ such that $\Vert f_1(t,\cdot) - f_2(t,\cdot) \Vert_{L^2(\Omega)} > 0$.
	By continuity in time we can find an interval $I$ of length $\delta > 0$ that contains $t$, and a constant $c>0$ such that
\begin{equation*}
	\Vert f_1(s,\cdot) - f_2(s,\cdot) \Vert_{L^2(\domain)} \geq c
\end{equation*}
	for all $s \in I$.
	However, weak solutions are unique in $L^\infty([0,T], L^2(\Omega, \Real))$, see \cite[Cor.~II.1]{DipLio89}, where we used the embedding of $\mathcal V$ into $C^1_0(\domain, \R^n)$.
	This yields a contradiction.
\end{proof}
Additionally, we can state and prove the following stability result for solutions of \eqref{eq:PDE}.
\begin{proposition}[Stability]\label{prop:stability_sol}
	Let $v_i \rightharpoonup v$ in $L^2([0, T],\mathcal V)$ and $f_i$ denote the weak solution of \cref{eq:PDE} corresponding to $v_i$.
	Then for every $t \in [0,T]$, there exists a subsequence, also denoted with $f_i$, such that $f_i(t,\cdot) \to f(t,\cdot)$ in $L^2(\domain, \R)$.
\end{proposition}
\begin{proof}
	The solution of \cref{eq:ODE} corresponding to $v_i$ is denoted by $X_i$.
	Fix an arbitrary $t\in[0,T]$.
	From \cref{thm:DiffeoVelo} we conclude $\Vert X_i(0,t,\cdot) - X(0,t,\cdot)\Vert_{C^0(\overline{\domain})} \to 0$.
	Further, \cite[Thm.~3.1.10]{Effland17} implies that $X_i(0,t,\cdot)$ is uniformly bounded for all $i\in \N$ in $C^{1,\alpha}(\overline{\domain})$, which implies $f_0(X_i(0,t,\cdot)) \to f_0(X(0,t,\cdot))$ in $L^2(\domain, \R)$ by \cite[Cor.~3]{NeuPerSte18}.
	
	It is left to show that a subsequence, also denoted by $X_i$, exists such that $X_i(0,t,\cdot) \to X(0,t,\cdot)$ in $C^1(\overline{\domain},\Real^n)$. This concludes the proof since it also implies the convergence of $\det (\Diff_x X_i(0,t,\cdot)) \to \det (\Diff_x X(0,t,\cdot))$ in $C^0(\overline{\domain})$.
	However, $X_i(0,t,\cdot)$ is uniformly bounded in $C^{1,\alpha}(\overline{\domain}, \R^n)$ and it follows that $\Diff_x X_i(0,t,\cdot)$ is uniformly bounded in $C^{0,\alpha}(\overline{\domain}, \R^{n\times n})$.
	By using the compact embedding of $C^{0,\alpha}(\overline{\domain}, \R^{n\times n})$ into $C^0(\overline{\domain}, \R^{n\times n})$ \cite[Lemma 6.33]{Gilbarg2015}, there exists a subsequence of $X_i(0,t,\cdot)$ that converges to $X(0,t,\cdot)$ in $C^{1}(\overline{\domain}, \R^n)$.
\end{proof}

\subsection{Transport equation with \texorpdfstring{$H^1$}{H1} regularity}
Here, we prove well-posedness of the transport equation that arises in the LDDMM framework using the geometric group action.
Compared to the previous section, the space regularity assumptions on $v$ are weaker and fit the setting in \cite{DipLio89}.

The transport equation reads as
\begin{equation}\label{eq:PDE2}
\begin{cases}
\dfrac{\partial}{\partial t} f(t,x) + v(t,x) \grad_x f(t,x)  = 0 \quad &\text{for $(t,x)\in [0,T]\times\Omega$,} \\[.5em] 
f(0,x) = f_0(x) \quad &\text{for $x\in\Omega$},
\end{cases}
\end{equation}
with coefficients
\begin{equation}
	v \in A \coloneqq \left\{v \in H^1([0,T] \times \Omega)^n \cap L^2([0,T],H^1_0(\Omega)^n)\colon \Vert \Div_x v \Vert_{L^\infty([0,T]\times \domain)} \leq C\right\}
\label{eq:seta}
\end{equation}
for some fixed constant $C$ and initial value $f_0 \in L^2(\Omega, \R)$.
The admissible set $A$ consists of all $H^1$ functions that are zero on the boundary of the spatial domain and have bounded divergence in the $L^\infty$ norm.

Note that the set $A$ is a subset of $H^1([0,T] \times \Omega)^n$, which is closed and convex so that it is a weakly closed subset of a reflexive Banach space.
In the following, we only check that $A$ is closed.
Let $v_i$ be a convergent sequence in $A$ with limit $v$.
Since the two involved spaces are Banach spaces, we only have to check that $v$ satisfies the constraint.
Assume that $\Vert \Div_x v \Vert_{L^\infty([0,T]\times \domain)} > C$, then there exists a set $B$ with positive measure and an $\epsilon>0$ such that for all $x\in B$ we have $\vert \Div_x v(x) \vert \geq C + \epsilon$.
Hence, we get $\Vert \Div_x v_i -\Div_x v\Vert_{L^2([0,T]\times \domain)}\geq \sqrt{\mu(B)}\epsilon$, which contradicts the convergence in $H^1$.

Again, equation \cref{eq:PDE2} has to be understood in the weak sense so that $f \in C^0([0,T], L^2(\Omega, \Real))$ is said to be a solution of \cref{eq:PDE2} if it satisfies
\begin{equation}\label{eq:SimpleWeakFormulation2}
\int_{0}^{T}\int_{\Omega} f(t,x) \Bigl(\Div_x \bigl(v(t,x) \eta(t,x)\bigr) + \frac{\partial}{\partial t} \eta(t,x) \Bigr) \dint x \dint t 
+ \int_{\Omega} f_0(x) \eta(0,x) \dint x = 0
\end{equation}
for all $\eta \in C^\infty_c([0,T) \times \Omega)$.
The next theorem is an existence and stability result, see \cite[Cors. II.1 and II.2, Thm.~II.5]{DipLio89}.
\begin{theorem}[Existence \& Stability]
	For every $v \in A$ there exists a unique weak solution $f \in C^0([0,T], L^2(\Omega, \R))$ of \cref{eq:PDE2}.
	If $v_i \in A$ converges to $v\in A$ in the norm of $L^2([0,T]\times \Omega, \R^n)$,
	then the corresponding sequence of  weak solutions $f_i \in C^0([0,T], L^2(\Omega, \R))$ converges to $f$ in $C^0([0,T], L^2(\Omega, \R))$.
\end{theorem}
\begin{proof}
	The existence and uniqueness of weak solutions follows from \cite[Corrs.~II.1 and II.2]{DipLio89}.
	Note that these solutions are also renormalised due to \cite[Thm.~II.3]{DipLio89}.
	
	We recast the second part of the theorem such that it has the exact form of \cite[Thm.~II.5]{DipLio89}.
	First, note that both the velocity fields and the initial condition can be extended to $\R^n$ by zero outside of $\Omega$ due to boundary condition of $A$.
	Due to the conditions on $v$, the weak formulation is equivalent to the one for the extension in the $\R^n$ setting.
	The uniform boundedness condition on $f_i$ is satisfied since $\domain$ is bounded.
\end{proof}
\begin{corollary}
	Let $v_i \rightharpoonup v\in A$ with the inner product of $H^1([0,T] \times \Omega)^n$.
	Then, $f_i$ converges to $f$ in $C^0([0,T], L^2(\Omega, \R))$.
\end{corollary}
\begin{proof}
	Combine the previous theorem with the compact embedding of $H^1([0,T] \times \Omega)^n$ into $L^2([0,T] \times \Omega)^n$ (Rellich embedding theorem \cite[A6.4]{Alt2012}).
\end{proof}
\begin{remark}
	Note that the same arguments can be used if we use higher spatial regularity, such as $H^2$, in this section. From a numerical point of view, the bound on the divergence is always satisfied for $C$ large enough if we use linear interpolation for the velocities on a fixed grid. Here we use that all norms are equivalent on finite dimensional spaces.
\end{remark}

\section{Regularising properties of template-based image reconstruction}
In this section, we prove regularising properties of template-based reconstruction as defined in \cref{eq:fun_to_min}.
Recall that the problem reads
\begin{equation*}
\begin{aligned}
\min_{v \in V} &\ D(K(f_{v}(T, \cdot)), g) + \gamma \Vert v \Vert_V^2, \\
\text{s.t.} & \ C(v),
\end{aligned}
\end{equation*}
where $C(v)$ is the Cauchy problem with either the transport or the continuity equation.
%Registration aims to deform a given template image such that it matches a given target image.
%In indirect registration, the target is only known in terms of an indirect observation through some operator $K\colon L^2(\domain, \R) \to Y$ and the corresponding data $g \in Y$.
%In order to rate the quality of the matching, we use a data discrepancy functional $D \colon Y \times Y \to \R_{\geq 0}$ on the data space.
%The indirect registration model is formulated in variational form, where the minimisation is over a suitable set of velocities in $V \coloneqq L^2\bigl([0,T],\mathcal V\bigr)$.
The admissible set $\mathcal V$ is chosen such that the regularity requirements stated in the previous section are satisfied.
%\begin{equation}
%	J_{\gamma, g} : v \in V \mapsto \gamma \Vert v \Vert_V^2 + D\bigl(K(f_v(T,\cdot)), g\bigr),
%\end{equation}
%where $f_v$ is the unique weak solution of the PDE \cref{eq:PDE} corresponding to the velocity $v$.
%The functional consists of the data fidelity term $D(K(f_v(T,\cdot)), g)$, a cost $\Vert v \Vert_V^2$ for deforming $f_0$ to $f_v(T,\cdot)$, and a regularisation parameter $\gamma > 0$, which controls the relative weight of the two terms.
%All of the following considerations also hold true for the PDE constraint \cref{eq:PDE2} instead of \cref{eq:PDE} if we choose $V = A$.
For the following considerations we require these assumptions on $K$ and $D$:
\begin{enumerate}
	\item The operator $K$ is continuous, $D(\cdot, g)$ is lower semi-continuous for each $g \in Y$, and $D(g, \cdot)$ is continuous for each $g \in Y$.
	\item  If $f_n,g_n$ are two convergent sequences with limits $f$ and $g$, respectively,
	then $D$ must satisfy $\liminf_{n \to \infty} D(f_n,g) \leq \liminf_{n \to \infty} D(f_n,g_n)$.
	\item If $D(f,g)=0$, then $f=g$.
\end{enumerate}
Note that the requirements on $D$ are satisfied if $D$ is a metric.
The obtained results are along the lines of \cite{CheOkt18} but are adapted to our setting and to our notation.
For simplicity we stick to the notation of the continuity equation, but want to mention that the same derivations hold for the transport equation with coefficients in the set $A$.
First, we prove that a minimiser of the problem exists.
\begin{proposition}[Existence]
	For every $f_0 \in L^2(\domain, \R)$, the functional $J_{\gamma, g}$ defined in \cref{eq:fun_to_min2} has a minimiser.
\end{proposition}
\begin{proof}
	The idea of the proof is to construct a minimising sequence which is weakly convergent and then use that the functional is weakly lower semi-continuous. Let us consider a sequence $v_n$ such that $J_{\gamma, g} (v_n)$ converges to $\inf_v J_{\gamma, g}(v)$. By construction of the functional, $v_n$ is bounded in $L^2([0,T], \mathcal{V})$ and hence there exists a subsequence, also denoted with $v_n$, such that $v_n \rightharpoonup v_\infty$.
	By \cref{prop:stability_sol}, there exists a subsequence, also denoted with $v_n$, such that $f_{v_n}(T,\cdot) \to f_{v_\infty}(T,\cdot)$ in $L^2(\domain, \R)$.
	With this at hand, we are able to prove weak lower semi-continuity of the data term.
	Indeed, as $K$ is continuous, from $f_{v_n}(T,\cdot) \to f_{v_\infty}(T,\cdot)$
	we get $K(f_{v_n}(T,\cdot)) \to K(f_{v_\infty}(T,\cdot))$.
	Since $D(\cdot, g)$ is lower semi-continuous, we obtain that 
	$D(K(f_{v_\infty}(T,\cdot)),g) \leq \liminf_{n \to \infty} D(K(f_{v_n}(T,\cdot)),g)$.
	This concludes the proof, since the whole functional is (weakly) lower semi-continuous, and hence $J_{\gamma,g}(v_\infty) \leq \inf_v J_{\gamma, g}(v)$.
\end{proof}
Next, we state a stability result.
\begin{proposition}[Stability]
	Let $f_0 \in L^2(\domain, \R)$ and $\gamma>0$.
	Let $g_n$ be a sequence in $Y$ converging to $g \in Y$.
	For each $n$, we choose $v_n$ as minimiser of $J_{\gamma, g_n}$.
	Then, there exists a subsequence of $v_n$ which weakly converges towards a minimiser $v$ of $J_{\gamma, g}$.
\end{proposition}
\begin{proof}
	By the properties of $D$ it holds, for every $n$, that
	\[\Vert v_n \Vert_V^2 \leq \frac{1}{\gamma} J_{\gamma, g_n} (v_n) \leq \frac{1}{\gamma} J_{\gamma, g_n} (0) = \frac{1}{\gamma} D(K(f_0), g_n) \rightarrow \frac{1}{\gamma} D(K(f_0), g) < \infty.\]
	Hence, $v_n$ is bounded in $L^2([0,T], \mathcal V)$ and there exists a subsequence, also denoted with $v_n$, such that $v_n \rightharpoonup v$.
	From the weak convergence we obtain
	$\gamma \Vert v \Vert_V^2 \leq \gamma \liminf_{n \to \infty} \Vert v_n \Vert_V^2$. 
	
	By passing to a subsequence and by using \cref{prop:stability_sol}, we deduce that $f_{v_n}(T,\cdot) \to f_{v}(T,\cdot)$.
	Together with the convergence of $g_n$ and the convergence property of $D$ this implies
	\[D(K(f_{v}(T,\cdot)), g) \leq \liminf_{n\to \infty} D(K(f_{v_n}(T,\cdot)), g) \leq \liminf_{n\to \infty} D(K(f_{v_n}(T,\cdot)), g_n).\]
	Thus, for any $\tilde v$, it holds that
	\[J_{\gamma, g} (v) \leq \liminf_{n \to \infty} \gamma \Vert v_n\Vert_V^2  + D(K(f_{v_n}(T,\cdot)), g_n) = \liminf_{n \to \infty} J_{\gamma, g_n} (v_n) \leq \liminf_{n \to \infty} J_{\gamma, g_n} (\tilde v),\]
	because $v_n$ minimises $J_{\gamma, g_n}$.
	Then, as $J_{\gamma, g_n} (\tilde v)$ converges to $J_{\gamma,g} (\tilde v)$ by the assumptions on $D$, we deduce $J_{\gamma, g} (v) \leq J_{\gamma,g} (\tilde v)$ and hence that $v$ minimises $J_{\gamma, g}$.
\end{proof}
Finally, we state a convergence result for the method.
\begin{proposition}[Convergence]
	Let $f_0 \in L^2(\domain, \R)$ and $g \in Y$, and suppose that there exists $\hat{v} \in L^2([0,T], \mathcal V)$ such that $K(f_{\hat v}(T,\cdot)) = g$.
	Further, assume that $\gamma : \R_{>0} \mapsto \R_{>0}$ satisfies $\gamma(\delta) \to 0$ and $\frac{\delta}{\gamma(\delta)} \to 0$ as $\delta \to 0$. 
	Now let $\delta_n$ be a sequence of positive numbers converging to 0 and assume that $g_n$ is a data sequence satisfying $D(g,g_n) \leq \delta_n$ for each $n$.
	Let $v_n$ be a minimiser of $J_{\gamma_n, g_n}$, where $\gamma_n = \gamma(\delta_n)$.
	Then, there exists a subsequence of $v_n$ which weakly converges towards an element $v$ such that
	$K(f_{v}(T,\cdot)) = g$.
\end{proposition}
\begin{proof}
	For every $n$, it holds that
	\begin{equation}
		\Vert v_n \Vert_V^2 \leq \frac{1}{\gamma_n} J_{\gamma_n, g_n} (v_n) \leq  \frac{1}{\gamma_n} J_{\gamma_n, g_n} (\hat{v}) =  \frac{1}{\gamma_n}  \bigr(D(g, g_n) + \gamma_n \Vert \hat{v} \Vert_V^2 \bigl)
		\leq \frac{\delta_n}{\gamma_n} + \Vert \hat{v} \Vert_V^2.
	\end{equation}
	From the requirements on $\gamma$ and $\delta$ we deduce that $v_n$ is bounded in $L^2 ([0,T], \mathcal V)$ and then that up to an extraction, $v_n$ weakly
	converges to some $v$ in $L^2 ([0,T], \mathcal V)$.
	
	Further, it holds $D(K(f_{v}(T,\cdot)), g) \leq \liminf_{n \to \infty} D(K(f_{v_n}(T,\cdot)), g_n)$ with the same arguments as in the previous proposition. Finally, for every $n$, it holds that
	\begin{equation}
		D(K(f_{v_n}(T,\cdot)), g_n) \leq J_{\gamma_n, g_n} (v_n) \leq J_{\gamma_n, g_n} (\hat{v}) = D(g, g_n) + \gamma_n \Vert \hat{v}\Vert_V^2,
	\end{equation}
	where the two rightmost terms both converge to zero.
	Thus, $K(f_{v}(T,\cdot)) = g$ by the assumptions on $D$.
\end{proof}
We conclude with a remark on data discrepancy functionals that satisfy the conditions and will be used in our numerical experiments in \cref{sec:experiments}.
\begin{remark}
	We now assume that the data space $Y$ is a real Hilbert space.
	Clearly, the conditions are satisfied if $D_{\mathrm{SSD}}(f,g) = \Vert f - g\Vert_{Y}^2$.
	We will only check the convergence condition.
	It holds
	\[\liminf_{n \to \infty}\Vert f_n - g\Vert_{Y}^2 = \liminf_{n \to \infty}\Vert f_n - g_n\Vert_{Y}^2 +2 \langle f_n - g_n,g_n-g \rangle + \Vert g - g_n\Vert_{Y}^2,\]
	where the last two terms converge to zero since convergent sequences are bounded.

	Another function that satisfies the conditions is $D_{\mathrm{NCC}} \colon Y\setminus\{0\} \times Y\setminus\{0\} \to [0,1]$ with
	\[D_{\mathrm{NCC}}(f,g) = 1 - \frac{\langle f,g \rangle^2}{\Vert f \Vert_{Y}^2 \Vert g \Vert_{Y}^2},\]
	which is based on the normalised cross correlation. First, note that $\tilde D(\cdot,g) = \frac{\langle \cdot,g \rangle^2}{\Vert g \Vert_{Y}^2}$ and the function $\Vert \cdot \Vert_{Y}^{-2}$ are continuous. Thus, we get that $D_{\mathrm{NCC}}(\cdot,g)$ is continuous.
	By symmetry, this also holds for $D_{\mathrm{NCC}}(g, \cdot)$.
	It remains to check the convergence property:
	\begin{align}
		\lim_{n \to \infty} 1 - D_{\mathrm{NCC}}(f_n,g) &= \lim_{n \to \infty} \frac{\left( \langle f_n,g -g_n \rangle + \langle f_n,g_n \rangle \right)^2}{\Vert f_n \Vert_{Y}^2 \Vert g \Vert_{Y}^2} = \lim_{n \to \infty} \frac{\langle f_n,g_n \rangle^2}{\Vert f_n \Vert_{Y}^2 \Vert g \Vert_{Y}^2}\\
		&= \lim_{n \to \infty} \frac{\langle f_n,g_n \rangle^2}{\Vert f_n \Vert_{Y}^2 \Vert g_n \Vert_{Y}^2} = \lim_{n \to \infty} 1 - D_{\mathrm{NCC}}(f_n,g_n).
	\end{align}
	From this we conclude $\liminf_{n \to \infty} D_{\mathrm{NCC}}(f_n,g) = \liminf_{n \to \infty} D_{\mathrm{NCC}}(f_n,g_n)$.
	Unfortunately, $D_{\mathrm{NCC}}(f,g) =0$ only implies $f = c g$, with $c \in \R$.
\end{remark}

\section{Numerical solution}

The focus of this section is to approximately solve problem \cref{eq:fun_to_min}.
Our approach is based on the Lagrangian methods developed in~\cite{ManRut17} and the inexact multilevel Gauss--Newton method used in~\cite{Mod09}.
Both methods and their necessary modifications are briefly outlined here.

As customary in PDE-constrained optimisation \cite[Chap.~3]{DeL15}, we eliminate the state equation by defining a \emph{control-to-state} operator, which parametrises the final state $f_{v}(T, \cdot)$ in terms of the unknown velocities $v$.
With a slight abuse of notation, we define this solution map as
\begin{equation}
\begin{aligned}
	S\colon V & \to L^2(\Omega, \R), \\
	v & \mapsto f_{v}(T, \cdot) \eqqcolon f(v).
\end{aligned}
\label{eq:solutionmap}
\end{equation}
Here, $f_{v}$ denotes the unique solution to either the transport or the continuity equation, as defined in \cref{sec:background}.
As a result, we obtain the reduced form of \cref{eq:fun_to_min}:
\begin{equation}
	\min_{v \in V} \, D(K(f(v)), g) + \gamma R(v).
\label{eq:reducedfunctional}
\end{equation}
Here, $R\colon V \to \R_{\ge 0}$ is a regularisation functional that can be written as
\begin{equation}
	R(v) = \frac{1}{2} \int_{0}^{T} \int_{\Omega} \lVert Bv(t, x) \rVert^{2} \dint x \dint t
\label{eq:regfunctional}
\end{equation}
with $B$ denoting a linear (vectorial) differential operator.

In this work, we consider the operators $B = \nabla_{x}$ and $B = \Delta_{x}$, which are also used in \cite{ManRut17}.
We refer to the resulting functionals $R$ as \emph{diffusion} and \emph{curvature regularisation} functionals, respectively.
Note that $B$ can as well be chosen to incorporate derivatives with respect to time.

Amongst the operators above, we also consider a regularisation functional that resembles the norm of the space $V = L^{2}([0, T], H_{0}^{3}(\Omega, \R^{n}))$.
This particular choice is motivated by the fact that, for $n = \{2, 3\}$, the space $H_{0}^{3}(\Omega, \R^{n})$ can be continuously embedded in $C_{0}^{1, \alpha}(\Omega, \R^{n})$, for some $\alpha > 0$, so that the results in \cref{sec:background} hold.
The norm of $V$ is given by
\begin{equation}
	\lVert v \rVert_{V}^{2} = \frac{1}{2} \int_{0}^{T} \lVert v(t, \cdot) \rVert_{L^{2}(\Omega, \R^{n})}^{2} \, \dint t + \frac{1}{2} \int_{0}^{T} \lvert v(t, \cdot) \rvert_{H^{3}(\Omega, \R^{n})}^{2} \, \dint t.
\label{eq:normv}
\end{equation}
Here, $\lvert \cdot \rvert_{H^{k}(\Omega, \R^{n})}$ denotes the usual $H^{k}$-seminorm including only the highest-order partial derivatives.
By the Gagliardo--Nirenberg inequality, \cref{eq:normv} is equivalent to the usual norm of $L^{2}([0, T], H_{0}^{3}(\Omega, \R^{n}))$.
To simplify numerical optimisation, we omit the requirement that $v$ is compactly supported in $\Omega$ and minimise over $L^{2}([0, T], H^{3}(\Omega, \R^{n}))$.

In order to solve problem \cref{eq:reducedfunctional}, we follow a discretise-then-optimise strategy.
Without loss of generality, we assume that the domain is $\Omega = (0, 1)^{n}$.
We partition it into a regular grid consisting of $m^{n}$ equally sized cells of edge length $h_{X} = 1 / m$ in every coordinate direction.

The template image $f_{0} \in L^2(\Omega, \R)$ is assumed to be sampled at cell-centred locations $\mathbf{x}_{c} \in \R^{m^{n}}$, giving rise to its discrete version $\mathbf{f}_{0}(\mathbf{x}_{c}) \in \R^{m^{n}}$.
The template image is interpolated on the cell-centred grid by means of cubic B-spline interpolation as outlined in \cite[Chap.~3.4]{Mod09}.

Similarly, the time domain is assumed to be $[0, 1]$ and is partitioned into $m_{t}$ equally sized cells of length $h_{t} = 1 / m_{t}$.
We assume that the unknown velocities $v\colon [0, 1] \times \Omega \to \R^{n}$ are sampled at cell-centred locations in space as well as at cell-centred locations in time, leading to a vector of unknowns $\mathbf{v} \in \R^{N}$, where $N = (m_{t} + 1) \cdot n \cdot m^{n}$ is the total number of unknowns of the finite-dimensional minimisation problem.

\paragraph{Lagrangian solver}

In order to compute the solution map $f(v)$ numerically, i.e.~to solve the hyperbolic PDEs \cref{eq:PDE} and \cref{eq:PDE2}, the Lagrangian solver in~\cite{ManRut17} follows a two-step approach.
First, given a vector $\mathbf{v} \in \R^{N}$ of velocities, the ODE \cref{eq:ODE} is solved approximately using a fourth-order Runge--Kutta (RK4) method with $N_{t}$ equally spaced time steps of size $\Delta t$.
For simplicity, we follow the presentation in~\cite{ManRut17} based on an explicit first-order Euler method and refer to \cite[Sec.~3.1]{ManRut17} for the full details.

Given initial points $\mathbf{x} \in \R^{m^{n}}$ and velocities $\mathbf{v} \in \R^{N}$, an approximation $\mathbf{X}_{\mathbf{v}}\colon [0, 1]^{2} \times \R^{m^{n}} \to \R^{m^{n}}$ of the solution $X_{v}$ is given recursively by
\begin{equation}
	\mathbf{X}_{\mathbf{v}}(0, t_{k + 1}, \mathbf{x}) = \mathbf{X}_{\mathbf{v}}(0, t_{k}, \mathbf{x}) + \Delta t \, \mathbf{I}(\mathbf{v}, t_{k}, \mathbf{X}_{\mathbf{v}}(0, t_{k}, \mathbf{x})),
\label{eq:odesolve}
\end{equation}
for all $k = 0, 1, \dots, N_{t} - 1$, with initial condition $\mathbf{X}_{\mathbf{v}}(0, 0, \mathbf{x}) = \mathbf{x}$.
Here, $\mathbf{I}(\mathbf{v}, t_{k}, \mathbf{X}_{\mathbf{v}}(0, t_{k}, \mathbf{x}))$ denotes a componentwise interpolation of $\mathbf{v}$ at time $t_{k} = k \Delta t$ and at the points $\mathbf{X}_{\mathbf{v}}(0, t_{k}, \mathbf{x})$.
Note that, since the characteristic curves for both PDEs coincide, this step is identical regardless of which PDE we impose.

The second step computes approximate intensities of the final state $f_{v}(1, \cdot)$.
This step depends on the particular PDE.
For the transport equation, in order to compute the intensities at the grid points $\mathbf{x}_{c}$, we follow characteristic curves backwards in time, which is achieved by setting $\Delta t = - 1 / N_{t}$ in \cref{eq:odesolve}.
The deformed template is then given by
\begin{equation}
	\mathbf{f}(\mathbf{v}) = \mathbf{f}_{0}(\mathbf{X}_{\mathbf{v}}(1, 0, \mathbf{x}_{c})),
\label{eq:solmaptransp}
\end{equation}
where $\mathbf{f}_{0} \in \R^{m^{n}}$ is the interpolation of the discrete template image.

For the continuity equation,~\cite{ManRut17} proposes to use a \emph{particle-in-cell (PIC)} method, see \cite{CheKur06} for details.
The density of particles which are initially located at grid points $\mathbf{x}_{c}$ is represented by a linear combination of basis functions, which are then shifted by following the characteristics computed in the first step.
To determine the final density at grid points, exact integration over the grid cells is performed.
By setting $\Delta t = 1 / N_{t}$ in \cref{eq:odesolve}, the transformed template can be computed as
\begin{equation}
	\mathbf{f}(\mathbf{v}) = \mathbf{F}(\mathbf{X}_{\mathbf{v}}(0, 1, \mathbf{x}_{c})) \mathbf{f}_{0}(\mathbf{x}_{c}),
\label{eq:solmapcont}
\end{equation}
where $\mathbf{F} \in \R^{N \times N}$ is the pushforward matrix that computes the integrals over the shifted basis functions.
See \cite[Sec.~3.1]{ManRut17} for its detailed specification using linear, compactly supported basis functions.
By design, the method is mass-preserving at the discrete level.

\paragraph{Numerical optimisation}

Let us denote by $\mathbf{K}\colon \R^{N} \to \R^{M}$, $M \in \N$, a finite-dimensional, Fr\'{e}chet differentiable approximation of the (not necessarily linear) operator $K\colon L^2(\Omega, \R) \to Y$.
With the application to CT in mind, we will outline a discretisation of \cref{eq:reducedfunctional} suitable for the $n$-dimensional Radon transform, which maps a function on $\R^{n}$ into the set of its integrals over the hyperplanes in $\R^{n}$ \cite[Chap.~2]{Nat01}.

An element $K(f(v)) \in Y$ is a function on the unit cylinder $S^{n - 1} \times \R$ of $\R^{n + 1}$, where $S^{n - 1}$ is the $(n - 1)$-dimensional unit sphere.
We discretise this unit cylinder as follows.
First, we sample $p \in \N$ directions from $S^{n - 1}$.
When $n = 2$, as it is the case in our experiments in \cref{sec:experiments}, directions are parametrised by angles from the interval $[0, 180]$ degrees.
For simplicity, we say (slightly imprecise) that we take one measurement in each direction.
Second, similarly to the sampling of $\Omega$, we use an interval $(0, 1)$ instead of $\R$ and partition it into $q$ equally sized cells of length $h_{Y} = 1 / q$.
Depending on $n$ and the diameter of $\Omega$, the interval length may require adjustment. 
Each measurement $i$ is then sampled at cell-centred points $\mathbf{y}_{c} \in \R^{q}$ and denoted by $\mathbf{g}_{i}(\mathbf{y}_{c}) \in \R^{q}$.
All measurements are then concatenated into a vector $\mathbf{g} \coloneqq \mathbf{g}(\mathbf{y}_{c}) \in \R^{M}$, where $M = p \cdot q$.

The finite-dimensional optimisation problem in abstract form is then given by
\begin{equation}
	\min_{\mathbf{v} \in \R^{N}} \, \{ J_{\gamma, \mathbf{g}}(\mathbf{v}) \coloneqq D(\mathbf{K}(\mathbf{f}(\mathbf{v})), \mathbf{g}) + \gamma R(\mathbf{v}) \},
\label{eq:discreteproblem}
\end{equation}
where $D$ and $R$ are chosen to be discretisations of a distance and of \cref{eq:regfunctional}, respectively.

In further consequence, we approximate integrals using a midpoint quadrature rule.
As we are mainly interested in the setting where only few directions are given, we disregard integration over the unit sphere.
For vectors $\mathbf{x}, \mathbf{y} \in \R^{M}$, the corresponding approximations of the distance based on sum-of-squared-differences and the normalised cross correlation-based distance are then
\begin{equation}
	D_{\mathrm{SSD}}(\mathbf{x}, \mathbf{y}) \approx \frac{h_{Y}}{2} (\mathbf{x} - \mathbf{y})^{\top}(\mathbf{x} - \mathbf{y}) \quad \text{and} \quad D_{\mathrm{NCC}}(\mathbf{x}, \mathbf{y}) \approx 1 - \frac{(\mathbf{x}^{\top}\mathbf{y})^{2}}{\Vert \mathbf{x} \Vert^{2} \Vert \mathbf{y} \Vert^{2}},
\label{eq:distances}
\end{equation}
respectively.
See \cite[Chaps.~6.2 and 7.2]{Mod09} for details.
Note that, due to cancellation, no (spatial) discretisation parameter occurs in the approximation of the NCC above.

Moreover, we approximate the regularisation functional in \cref{eq:regfunctional} with
\begin{equation}
	R(\mathbf{v}) \approx \frac{h_{t} h_{X}^{n}}{2} \mathbf{v}^{\top} \mathbf{B}^{\top} \mathbf{B} \mathbf{v},
\label{eq:normvapprox}
\end{equation}
where $\mathbf{B} \in \R^{N \times N}$ is a finite-difference discretisation of the differential operator in \cref{eq:regfunctional}, analogous to \cite[Chap.~8.5]{Mod03}.
In our implementation we use zero Neumann boundary conditions and pad the spatial domain to mitigate boundary effects arising from the discretisation of the operator.

In order to apply (inexact) Gauss--Newton optimisation to problem \cref{eq:discreteproblem}, we require first- and (approximate) second-order derivatives of $J_{\gamma, \mathbf{g}}(\mathbf{v})$.
By application of the chain rule, we obtain
\begin{equation*}
	\frac{\partial}{\partial \mathbf{v}} J_{\gamma, \mathbf{g}}(\mathbf{v}) = \frac{\partial}{\partial \mathbf{v}} \mathbf{f}(\mathbf{v})^{\top} \frac{\partial}{\partial \mathbf{f}} \mathbf{K}(\mathbf{f}(\mathbf{v}))^{\top} \frac{\partial}{\partial \mathbf{x}} D(\mathbf{K}(\mathbf{f}(\mathbf{v})), \mathbf{g}) + \gamma \frac{\partial}{\partial \mathbf{v}} R(\mathbf{v}),
\end{equation*}
where $\partial \mathbf{K} / \partial \mathbf{f}$ is the Fr\'{e}chet derivative of $\mathbf{K}$ and $\partial \mathbf{f}(\mathbf{v}) / \partial \mathbf{v}$ is the derivative of the solution map \cref{eq:solutionmap} with respect to the velocities, which is given below.

The partial derivatives of the distance functions \cref{eq:distances} with respect to its first argument are given by
\begin{equation}
	\frac{\partial}{\partial \mathbf{x}} D_{\mathrm{SSD}}(\mathbf{x}, \mathbf{y}) = h_{Y}(\mathbf{x} - \mathbf{y}) \quad \text{and} \quad \frac{\partial^{2}}{\partial \mathbf{x}^{2}} D_{\mathrm{SSD}}(\mathbf{x}, \mathbf{y}) = h_{Y} \mathbf{I}_{N},
\label{eq:ssdderiv}
\end{equation}
where $\mathbf{I}_{N} \in \R^{N \times N}$ is the identity matrix of size $N$, and
\begin{equation*}
	\frac{\partial}{\partial \mathbf{x}} D_{\mathrm{NCC}}(\mathbf{x}, \mathbf{y}) = - \frac{2 (\mathbf{x}^{\top} \mathbf{y}) \mathbf{y}}{\lVert \mathbf{x} \rVert^{2} \lVert \mathbf{y} \rVert^{2}} + \frac{2 (\mathbf{x}^{\top} \mathbf{y})^{2} \mathbf{x}}{\lVert \mathbf{x} \rVert^{4} \lVert \mathbf{y} \rVert^{2}}
\end{equation*}
respectively.
Moreover, the derivatives of \cref{eq:normvapprox} are given by
\begin{equation*}
	\frac{\partial}{\partial \mathbf{v}} R(\mathbf{v}) = h_{t} h_{X}^{n} \mathbf{B}^{\top} \mathbf{B} \mathbf{v} \quad \text{and} \quad \frac{\partial^{2}}{\partial \mathbf{v}^{2}} R(\mathbf{v}) = h_{t} h_{X}^{n} \mathbf{B}^{\top} \mathbf{B}.
\end{equation*}

In order to obtain an efficient iterative second-order method for solving \cref{eq:discreteproblem}, one requires an approximation of the Hessian $\mathbf{H} \in \R^{N \times N}$ that balances the following tradeoff.
Ideally, it is reasonably efficient to compute, consumes limited memory (sparsity is desired), and has sufficient structure so that preconditioning can be used.
However, each iteration of the Gauss--Newton method should also provide a suitable descent direction.
For these reasons, we approximate the Hessian by
\begin{multline*}
	\mathbf{H}(\mathbf{v}) = \frac{\partial^{2}}{\partial \mathbf{v}^{2}} J_{\gamma, \mathbf{g}}(\mathbf{v}) \approx \frac{\partial}{\partial \mathbf{v}} \mathbf{f}(\mathbf{v})^{\top} \frac{\partial}{\partial \mathbf{f}} \mathbf{K}(\mathbf{f}(\mathbf{v}))^{\top} \frac{\partial^{2}}{\partial \mathbf{x}^{2}} D(\mathbf{K}(\mathbf{f}(\mathbf{v})), \mathbf{g}) \frac{\partial}{\partial \mathbf{f}} \mathbf{K}(\mathbf{f}(\mathbf{v})) \frac{\partial}{\partial \mathbf{v}} \mathbf{f}(\mathbf{v}) \\
	\qquad + \gamma h_{t} h_{X}^{n} \mathbf{B}^{\top} \mathbf{B} + \epsilon \mathbf{I}_{N},
\end{multline*}
where $\epsilon > 0$ ensures positive semidefiniteness.
For simplicity, the term involving $\partial^{2} \mathbf{f}(\mathbf{v}) / \partial \mathbf{v}^{2}$ is omitted and, regardless of the chosen distance, we use the second derivative in \cref{eq:ssdderiv} as an approximation of $\partial^{2} D(\mathbf{x}, \mathbf{y}) / \partial \mathbf{x}^{2}$.
In our numerical experiments we found that this choice works well for the problem considered in \cref{sec:experiments}.

It remains to discuss the derivative of the solution map.
For the transport equation, the application of the chain rule to \cref{eq:solmaptransp} yields
\begin{equation*}
	\frac{\partial}{\partial \mathbf{v}} \mathbf{f}(\mathbf{v}) = \nabla_{x} \mathbf{f}_{0}(\mathbf{X}_{\mathbf{v}}(1, 0, \mathbf{x}_{c})) \frac{\partial}{\partial \mathbf{v}} \mathbf{X}_{\mathbf{v}}(1, 0, \mathbf{x}_{c}),
\end{equation*}
where $\nabla_{x} \mathbf{f}_{0}$ denotes the gradient of the interpolation of the template image and $\partial \mathbf{X}_{\mathbf{v}} / \partial \mathbf{v}$ is the derivative of the endpoints of the characteristic curves with respect to the velocities, see below.
Similarly, for the solution map \cref{eq:solmapcont} that corresponds to the continuity equation, we obtain
\begin{equation*}
	\frac{\partial}{\partial \mathbf{v}} \mathbf{f}(\mathbf{v}) = \frac{\partial}{\partial \mathbf{X}_{\mathbf{v}}} (\mathbf{F}(\mathbf{X}_{\mathbf{v}}(0, 1, \mathbf{x}_{c})) \mathbf{f}_{0}(\mathbf{x}_{c})) \frac{\partial}{\partial \mathbf{v}} \mathbf{X}_{\mathbf{v}}(0, 1, \mathbf{x}_{c}).
\end{equation*}
Here, $\partial \mathbf{F} / \partial \mathbf{X}_{\mathbf{v}}$ is the derivative of the pushforward matrix with respect to the endpoints of the characteristics, again see \cite[Sec.~3.1]{ManRut17}.

If explicit time stepping methods are used to solve the ODE \cref{eq:ODE}, the partial derivative $\partial \mathbf{X}_{\mathbf{v}} / \partial \mathbf{v}$ can be computed recursively.
For example, for the forward Euler approach in \cref{eq:odesolve} it is given by
\begin{multline*}
	\frac{\partial}{\partial \mathbf{v}} \mathbf{X}_{\mathbf{v}}(0, t_{k + 1}, \mathbf{x}_{c}) = \frac{\partial}{\partial \mathbf{v}} \mathbf{X}_{\mathbf{v}}(0, t_{k}, \mathbf{x}_{c}) + \Delta t \, \frac{\partial}{\partial \mathbf{v}} \mathbf{I}(\mathbf{v}, t_{k}, \mathbf{X}_{\mathbf{v}}(0, t_{k}, \mathbf{x}_{c})) \\
	+ \Delta t \, \frac{\partial}{\partial \mathbf{X}_{\mathbf{v}}} \mathbf{I}(\mathbf{v}, t_{k}, \mathbf{X}_{\mathbf{v}}(0, t_{k}, \mathbf{x}_{c})) \frac{\partial}{\partial \mathbf{v}} \mathbf{X}_{\mathbf{v}}(0, t_{k}, \mathbf{x}_{c}),
\end{multline*}
for all $k = 0, 1, \dots, N_{t} - 1$, with $\partial \mathbf{I} / \partial \mathbf{v}$ and $\partial \mathbf{I} / \partial \mathbf{X}_{\mathbf{v}}$ being the derivatives of the interpolation schemes with respect to the velocities and with respect to the endpoints of the characteristics, respectively.
We refer to~\cite[Chap.~3.5]{Mod09} for details.
The case where characteristics are computed backwards in time can be handled similarly.

In order to solve the finite-dimensional minimisation problem \cref{eq:discreteproblem}, we apply a inexact Gauss--Newton--Krylov method, which proceeds as follows.
Given an initial guess $\mathbf{v}^{(0)} = \mathbf{0}$, we update the velocities in each iteration $i = 0, 1, \dots$ by $\mathbf{v}^{(i + 1)} = \mathbf{v}^{(i)} + \mu \delta \mathbf{v}$ until a termination criterion is satisfied.
Here, $\mu \in \R$ denotes a step size that is determined via Armijo line search and $\delta \mathbf{v} \in \R^{N}$ is the solution to the linear system
\begin{equation}
	\mathbf{H}(\mathbf{v}^{(i)}) \delta \mathbf{v} = - \frac{\partial}{\partial \mathbf{v}} J_{\gamma, \mathbf{g}}(\mathbf{v}^{(i)}).
\label{eq:linsys}
\end{equation}
For details on the stopping criteria and the line search we refer to~\cite[Chap.~6.3.3]{Mod09}.
We solve the system \cref{eq:linsys} approximately by means of a preconditioned conjugate gradient (PCG) method, which can be implemented matrix-free whenever the derivative of $\mathbf{K}$ and its adjoint can be computed matrix-free.
See \cite[Sec.~3.2]{ManRut17} for further details on the preconditioning.

Due to the non-convexity of \cref{eq:reducedfunctional} and to speed up computation, we use a multilevel strategy in order to reduce the risk of ending up in a local minimum, see \cite{HabMod06}.
On each level, we use a subsampled version of the velocities that were computed on the previous, more coarser, discretisation as initial guess.

While standard image registration typically uses resampling of the template and the target image \cite[Chap.~3.7]{Mod09}, the approach described here requires multilevel versions of the operator $\mathbf{K}$ together with a suitable method for resampling the measurements $\mathbf{g}$.
We stress that, if these are not available, optimisation can as well just be performed on the finest discretisation level.

In the following, we assume that $\mathbf{K}$ is a discretisation of the Radon transform \cite{Nat01}, which is a linear operator, and outline a suitable procedure for creating multilevel versions of the operator and the measured data.
The former is easily achieved with a computational backend such as ASTRA~\cite{AarPalCanJanBle16, AarPalBeeAltBal15}, which allows to explicitly specify the number of grid cells used to discretise the measurement geometry.
For the sake of simplicity, we restrict the presentation here to the case where $n = 2$, i.e.~$\Omega \subset \R^{2}$, and $\mathbf{K}$ is linear.

Let us assume that the number of grid cells used to discretise $\Omega$ at the finest level is $m = 2^{\ell}$, $\ell \in \N$.
In our experiments, we set the number of grid cells of the one-dimensional measurement domain $(0, 1)$ at the current level $k \le \ell$ to $q^{(k)} = 1.5 \cdot 2^{(k)}$ and set the length of each cell to $h_{Y}^{(k)} = 1 / q^{(k)}$.
Then, a multilevel representation of each measurement $\mathbf{g}_{i}$, $i \le p$, at cell-centred grid points $\mathbf{y}_{j} = (j - 1/2) h_{Y}^{(k - 1)}$ is given by
\begin{equation*}
	\mathbf{g}_{i}^{(k - 1)}(\mathbf{y}_{j}) \coloneqq \left( \mathbf{g}_{i}^{(k)}(\mathbf{y}_{j}) + \mathbf{g}_{i}^{(k)}(\mathbf{y}_{j} + h_{Y}^{(k)}) \right) / 4,
\end{equation*}
where the denominator arises from averaging over two neighbouring grid points and dividing the edge length of the imaging domain $\Omega$ in each coordinate direction in half.
The approach can be extended to higher dimensions.

\section{Numerical examples} \label{sec:experiments}

\begin{figure}[t]
	\captionsetup[subfigure]{justification=centering}
	\centering
	\begin{subfigure}[t]{0.32\textwidth}
		\centering
		\includegraphics[width = 0.98\textwidth]{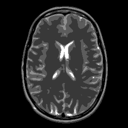}
		\caption{Template image.}
		\label{fig:BrainArtificialData:template}
	\end{subfigure}		
	\begin{subfigure}[t]{0.32\textwidth}
		\centering
		\includegraphics[width = 0.98\textwidth]{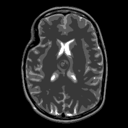}
		\caption{Unknown image.}
	\end{subfigure}	
	\begin{subfigure}[t]{0.32\textwidth}
		\centering
		\includegraphics[width = 0.98\textwidth]{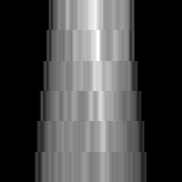}
		\caption{Measured (sinogram) data without noise.}
	\end{subfigure}	
	\caption{Synthetic example based on an artificial brain image \cite{GueLejPruUns12} that has been deformed manually. We generated six Radon transform measurements that correspond to six equally spaced angles from the interval $[0, 60]$ degrees.}
	\label{fig:BrainArtificialData}
\end{figure}

In our numerical experiments we use the Radon transform \cite{Nat01} as operator.
Other choices are possible and, assuming that one has access to a suitable resampling procedure for the measured data, the multilevel strategy can be applied as well.
The aim here is to investigate the reconstruction quality with different regularisation functionals, distances, and noise levels for both PDE constraints.
We show synthetic examples for the settings $n = 2$ and $n = 3$, and non-synthetic examples for $n = 2$ using real X-ray tomography data.
In the synthetic case, all shown reconstructions were computed from measurements taken from at most 10 directions (i.e.\ angles) sampled from intervals within $[0, 180]$ degrees.

All computations were performed using an Intel Xeon E5-2630 v4 $2.2 \, \mathrm{GHz}$ server equipped with $128 \, \mathrm{GB}$ RAM and an NVIDIA Quadro P6000 GPU featuring $24 \, \mathrm{GB}$ of memory.
The GPU was only used for computing the Radon transform of 3D volumes.

Before we proceed, we give a brief idea of suitable parameter choices.
For the multilevel approach we used in each synthetic example $32\times32$ pixels at the coarsest level and $128\times128$ pixels at the finest level, i.e.\ $\ell = 7$.
The size of the reconstructed images in the nonsynthetic examples was $128\times128$.
Again, three levels were used.
In the synthetic 3D example the reconstructed volume was $32 \times 32 \times 32$ and the coarsest level was $8 \times 8 \times 8$.

We used time dependent velocity fields with only one time step, i.e.\ $n_{t} = 1$, since this keeps the computational cost reasonable and sufficed for our examples.
The characteristics were computed using five Runge--Kutta steps, i.e.\ $N_{t} = 5$.

The spatial regularisation parameter depends on the chosen regularisation functional and the noise level, and was chosen in the order of $10^{-3}$, $10^{0}$, and $10^{3}$ for third-order, curvature, and diffusion regularisation, respectively, in the noisefree case and using the NCC-based distance.
The temporal regularisation parameter is less sensitive and was chosen in the order of $10^2$.
Furthermore, the parameter corresponding to the norm of $L^{2}(\Omega, \R^{n})$ in \cref{eq:normv} was set to $10^{-6}$.

\begin{figure}[t]
	\captionsetup[subfigure]{justification=centering}
	\centering
	\begin{subfigure}[t]{0.24\textwidth}
		\centering
		\includegraphics[width = 0.98\textwidth]{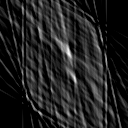}
		\caption{Reconstruction using filtered back-projection.}
		\label{fig:comparison:fbp}
	\end{subfigure}		
	\begin{subfigure}[t]{0.24\textwidth}
		\centering
		\includegraphics[width = 0.98\textwidth]{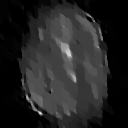}
		\caption{Reconstruction using $\mathcal{R}_1$ (TV reconstruction).}
		\label{fig:comparison:l2tv}
	\end{subfigure}	
	\begin{subfigure}[t]{0.24\textwidth}
		\centering
		\includegraphics[width = 0.98\textwidth]{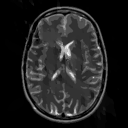}
		\caption{Reconstruction using $\mathcal{R}_2$ with given template.}
		\label{fig:comparison:l2tv2}
	\end{subfigure}
	\begin{subfigure}[t]{0.24\textwidth}
		\centering
		\includegraphics[width = 0.98\textwidth]{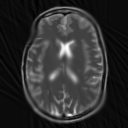}
		\caption{Reconstruction using the metamorphosis approach \cite{GriCheOkt18}.}
		\label{fig:comparison:metamorphosis}
	\end{subfigure}
	\caption{Comparison of different reconstruction models applied to an artificial brain image \cite{Mod09} that has been deformed manually. We generated six measurements that correspond to six equally spaced angles from the interval $[0, 60]$ degrees.}
	\label{fig:comparison}
\end{figure}

In our first example, we investigate different regularisation functionals with different noise levels together with the the transport equation.
The target is 2D Radon transform data based on a digital brain image and the template is a deformed version thereof, see \cref{fig:BrainArtificialData}.
Since we want to focus on the behaviour of the regularisation functionals, we do not treat the continuity equation here.
The data was generated using parallel beam tomography with only six equally distributed angles from the interval $[0, 60]$ degrees and was corrupted with Gaussian white noise of different levels.

\cref{fig:comparison} shows results obtained from the generated noisefree measurements using four existing methods.
In \cref{fig:comparison:fbp} filtered backprojection was used.
In \cref{fig:comparison:l2tv,fig:comparison:l2tv2}, the following two total variation regularisation-based models, see e.g.~\cite{CanRomTao06},
\begin{equation*}
	\min_{\mathbf{u}} \Vert \mathbf{K}\mathbf{u} - \mathbf{g} \Vert^{2} + \gamma \mathcal{R}_i(\mathbf{u}),
\end{equation*}
with $\mathcal{R}_1(\mathbf{u}) \coloneqq \mathrm{TV}(\mathbf{u})$, $\mathcal{R}_2(\mathbf{u}) \coloneqq \mathrm{TV}(\mathbf{u} - \mathbf{f}_{0})$, and $\gamma > 0$ were used.
Here, $\mathcal{R}_2(\mathbf{u})$ incorporates template information.
Approximate minimisers of both functionals were computed using the primal-dual hybrid gradient method \cite{ChaPoc11}.
For the case of filtered backprojection, the standard MATLAB implementation was used.
The results in \crefrange{fig:comparison:fbp}{fig:comparison:l2tv2} highlight why more sophisticated methods, such as the proposed template-based approach, are necessary to obtain satisfying reconstructions in this setting, and illustrate the challenges when dealing with very sparse data.

As outlined in \cref{sec:intro}, one possibility is the metamorphosis approach \cite{GriCheOkt18}.
In \cref{fig:comparison:metamorphosis} we show a result obtained with this method using the recommended parameters.
However, 200 iterations of gradient descent were performed, and the regularisation parameters were set to $\gamma = 10^{-5}$ and $\tau = 1$.
Observe the change in image intensities compared to \cref{fig:BrainArtificialData:template} and the blur in the heavily deformed regions.

In \cref{fig:BrainArtificial}, we show results for the different noise levels and different regularisation functionals computed with our approach.
All results were obtained using the NCC-based distance.
As expected, the quality of the reconstruction gets worse for higher noise levels and, consequentially, larger regularisation parameters were necessary.
Since data is acquired from only six directions, the influence of the noise is very strong.
Especially for the diffusive regularisation we needed to choose large regularisation parameters for higher noise levels, see \cref{fig:BrainError}.
Since diffusion corresponds to first-order regularisation, it is much easier to reconstruct the noise with ``rough'' deformations.
Overall, we found that second- and third-order regularisation performed similar when appropriate regularisation parameters were chosen.
Even though some theoretical results only hold for higher-order regularity, second-order regularisation seems sufficient for our use case.
The computation time for the results in \cref{fig:BrainArtificial} was between 200 and 700 seconds.

\begin{figure}
	\captionsetup[subfigure]{justification=centering}
	\centering	
	\begin{subfigure}[t]{0.32\textwidth}
		\centering
		\includegraphics[width = 0.98\textwidth]{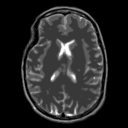}
		\caption{Diffusion regularisation, no noise, SSIM 0.920.}
	\end{subfigure}
	\begin{subfigure}[t]{0.32\textwidth}
		\centering
		\includegraphics[width = 0.98\textwidth]{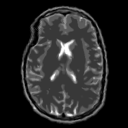}
		\caption{Diffusion regularisation, 5~\% noise, SSIM 0.867.}
	\end{subfigure}
	\begin{subfigure}[t]{0.32\textwidth}
		\centering
		\includegraphics[width = 0.98\textwidth]{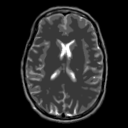}
		\caption{Diffusion regularisation, 10~\% noise, SSIM 0.798.}
		\label{fig:BrainError}
	\end{subfigure}

	\begin{subfigure}[t]{0.32\textwidth}
		\centering
		\includegraphics[width = 0.98\textwidth]{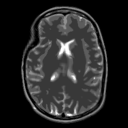}
		\caption{Curvature regularisation, no noise, SSIM 0.955.}
	\end{subfigure}
	\begin{subfigure}[t]{0.32\textwidth}
		\centering
		\includegraphics[width = 0.98\textwidth]{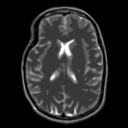}
		\caption{Curvature regularisation, 5~\% noise, SSIM 0.897.}
	\end{subfigure}
	\begin{subfigure}[t]{0.32\textwidth}
		\centering
		\includegraphics[width = 0.98\textwidth]{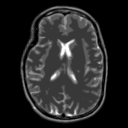}
		\caption{Curvature regularisation, 10~\% noise, SSIM 0.823.}
	\end{subfigure}

	\begin{subfigure}[t]{0.32\textwidth}
		\centering
		\includegraphics[width = 0.98\textwidth]{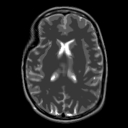}
		\caption{Third-order regularisation, no noise, SSIM 0.950.}
	\end{subfigure}
	\begin{subfigure}[t]{0.32\textwidth}
		\centering
		\includegraphics[width = 0.98\textwidth]{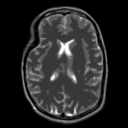}
		\caption{Third-order regularisation, 5~\% noise, SSIM 0.901.}
	\end{subfigure}
	\begin{subfigure}[t]{0.32\textwidth}
		\centering
		\includegraphics[width = 0.98\textwidth]{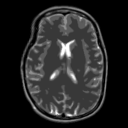}
		\caption{Third-order regularisation, 10~\% noise, SSIM 0.798.}
	\end{subfigure}
	\caption{Reconstructions for the artificial brain image in \cref{fig:BrainArtificialData} using our method and different regularisation functionals. Note that only six measurements were used. The measured data was corrupted with noise of different levels.}
	\label{fig:BrainArtificial}
\end{figure}

In the second example, see \cref{fig:Hands}, we compare the behaviour of the SSD and the NCC-based distance.
The example consists of two different hands which, in addition, are rotated relative to each other.
Here, the deformation is much larger than in the previous example, but still fairly regular.
The data was generated similarly to the previous example, but with only five angles from the interval $[0, 75]$ degrees.
Note also that the intensities of the template and target image are different (roughly by a factor of two).
First, we discuss the transport equation.
The intensity difference is a serious issue if we use the SSD distance, as we can see in \cref{fig:HandsError}.
The hand is deformed into a smaller version in order to compensate the differences. If we use the NCC-based distance instead, which can deal with such discrepancies, the result is much better from a visual point of view.
The shapes are well-aligned.
The resulting SSIM value is still low, which is not surprising since SSIM is not invariant with respect to intensity differences between perfectly aligned images.
However, neither of the two approaches is able to remove or create any of the additional (noise) structures in the images.
For the combination SSD with continuity equation, no satisfactory results could be obtained.
Since no change of intensity is possible by changing the size of the hand, part of it is moved outside of the image.
This behaviour could potentially be corrected if other boundary conditions are used in the implementation.
Therefore, we do not provide an example image for this case.
Using the NCC-based distance, the results look similar as for the transport equation with slightly worse SSIM value.
These results suggest that the NCC-based distance is a more robust choice that avoids unnatural deformations, which would be required in the case of SSD to compensate intensity differences.
In this example, the computation time was between 50 and 325 seconds.
 
\begin{figure}[t]
	\captionsetup[subfigure]{justification=centering}
	\centering
	\begin{subfigure}[t]{0.32\textwidth}
		\centering
		\includegraphics[width = 0.98\textwidth]{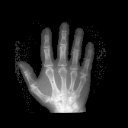}
		\caption{Template image.}
	\end{subfigure}		
	\begin{subfigure}[t]{0.32\textwidth}
		\centering
		\includegraphics[width = 0.98\textwidth]{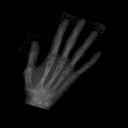}
		\caption{Unknown image.}
	\end{subfigure}	
	\begin{subfigure}[t]{0.32\textwidth}
		\centering
		\includegraphics[width = 0.98\textwidth]{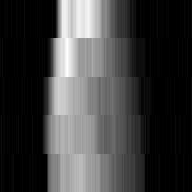}
		\caption{Measured noisy (sinogram) data.}
	\end{subfigure}	
	
	\begin{subfigure}[t]{0.32\textwidth}
		\centering
		\includegraphics[width = 0.98\textwidth]{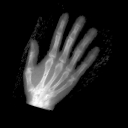}
		\caption{NCC-based distance with transport equation, SSIM 0.562.}
	\end{subfigure}
	\begin{subfigure}[t]{0.32\textwidth}
		\centering
		\includegraphics[width = 0.98\textwidth]{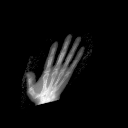}
		\caption{SSD distance with transport equation, SSIM 0.568.}
		\label{fig:HandsError}
	\end{subfigure}
	\begin{subfigure}[t]{0.32\textwidth}
		\centering
		\includegraphics[width = 0.98\textwidth]{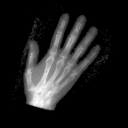}
		\caption{NCC-based distance with continuity equation, SSIM 0.555.}
	\end{subfigure}

	\caption{Reconstructions of manually deformed Hand \cite{Mod09} images with different image intensity levels using our method. We generated five measurements that correspond to five equally spaced angles from the interval $[0, 75]$ degrees and added five percent noise.}
	\label{fig:Hands}	
\end{figure}

In the next example, see \cref{fig:HNSP}, we compare the continuity equation with the transport equation as constraint together with the NCC-based distance.
The continuity equation allows for limited change of mass along the deformation path.
Since the intensity change scales with the determinant of the Jacobian, bigger changes are only possible if areas are compressed or extended a lot.
In the presented example this occurs only to a mild extent.
For this example, the continuity equation and the transport equation yield visually similar results with minor differences in the SSIM value.
As in the previous examples, higher-order regularisation is beneficial and artefacts occur for the diffusion regularisation.
The computation time amounted to roughly 64 to 360 seconds in this example.

\begin{figure}[h]
	\captionsetup[subfigure]{justification=centering}
	\centering
	\begin{subfigure}[t]{0.32\textwidth}
		\centering
		\includegraphics[width = 0.98\textwidth]{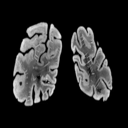}
		\caption{Template image.}
	\end{subfigure}		
	\begin{subfigure}[t]{0.32\textwidth}
		\centering
		\includegraphics[width = 0.98\textwidth]{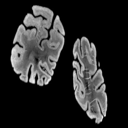}
		\caption{Unknown image.}
	\end{subfigure}	
	\begin{subfigure}[t]{0.32\textwidth}
		\centering
		\includegraphics[width = 0.98\textwidth]{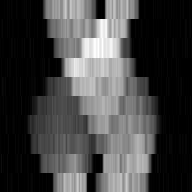}
		\caption{Measured noisy (sinogram) data.}
	\end{subfigure}	
	
	\begin{subfigure}[t]{0.32\textwidth}
		\centering
		\includegraphics[width = 0.98\textwidth]{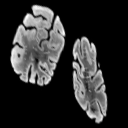}
		\caption{Continuity equation with third-order regularisation, SSIM 0.910.}
	\end{subfigure}
	\begin{subfigure}[t]{0.32\textwidth}
		\centering
		\includegraphics[width = 0.98\textwidth]{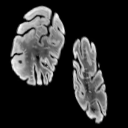}
		\caption{Continuity equation with curvature regularisation, SSIM 0.753.}
	\end{subfigure}
	\begin{subfigure}[t]{0.32\textwidth}
		\centering
		\includegraphics[width = 0.98\textwidth]{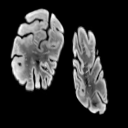}
		\caption{Continuity equation with diffusion regularisation, SSIM 0.560.}
	\end{subfigure}
	
	\begin{subfigure}[t]{0.32\textwidth}
		\centering
		\includegraphics[width = 0.98\textwidth]{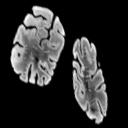}
		\caption{Transport equation with third-order regularisation, SSIM 0.913.}
	\end{subfigure}
	\begin{subfigure}[t]{0.32\textwidth}
		\centering
		\includegraphics[width = 0.98\textwidth]{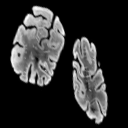}
		\caption{Transport equation with curvature regularisation, SSIM 0.912.}\label{fig:BadGeo}
	\end{subfigure}
	\begin{subfigure}[t]{0.32\textwidth}
		\centering
		\includegraphics[width = 0.98\textwidth]{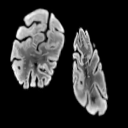}
		\caption{Transport equation with diffusion regularisation, SSIM 0.580.}
	\end{subfigure}
	\caption{Reconstructions for the HNSP \cite{Mod09} image using our approach, different regularisation functionals, and different PDE constraints. Here, ten measurements corresponding to ten angles equally distributed in the interval $[0, 180]$ degrees were taken. The measured data was corrupted with five percent noise.}
	\label{fig:HNSP}
\end{figure}

In \cref{fig:circles}, we created an artificial pair of images consisting of a disk to show the possibilities of intensity changes when using the continuity equation as a constraint.
Both template and unknown image were constructed so that their total mass is equal.
The measurements were generated as before using only five angles uniformly distributed in the interval $[0, 90]$ degrees.
Furthermore, we used curvature regularisation.
For the transport equation we observe that the shape is matched, but the intensity is not correct, see \cref{fig:circlesgeo}.
If we use the continuity equation instead, intensity changes are possible, which can be observed in \cref{fig:circlesmp}.
The computation time for the two results was 90 and 500 seconds.

In order to demonstrate the practicality of our method, we computed results from nonsynthetic X-ray tomography data \cite{BubHauHuoRimSil16, HamHarKalKujNie15}, which are available online.\footnote{\url{https://doi.org/10.5281/zenodo.1254204}}\textsuperscript{,}\footnote{\url{https://doi.org/10.5281/zenodo.1254206}}
See \Cref{fig:realdata} for these two examples (`lotus' and `walnut').
The template was generated by applying filtered backprojection to the full measurements and by subsequently deforming it.
Then, this deformed templated was used in our method to compute a reconstruction from only few measurement directions.
The computation time amounted to roughly 80 and 600 seconds in these examples.
In both nonsynthetic examples the use of the NCC-based distance proved crucial and no satisfactory result could be obtained using SSD.

In \cref{fig:mice3D}, we demonstrate that our framework is also capable of reconstructing 3D volumes.
Here, we used the SSD distance together with curvature regularisation and the transport equation.
We applied the 3D Radon transform to obtain ten measurements from angles within $[0, 180]$. 
The total computation time was roughly 800 seconds.

All in all, our results demonstrate that, given a suitable template image, very reasonable reconstructions can efficiently be obtained from only a few measurements, even in the presence of noise.
Moreover, our examples show that the NCC-based distance adds robustness to the approach with regard to discrepancies in the image intensities.

\begin{figure}[t]
	\captionsetup[subfigure]{justification=centering}
	\centering
	\begin{subfigure}[t]{0.19\textwidth}
		\centering
		\includegraphics[width = 0.99\textwidth]{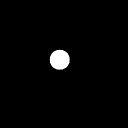}
		\caption{Template image.}
	\end{subfigure}		
	\begin{subfigure}[t]{0.19\textwidth}
		\centering
		\includegraphics[width = 0.99\textwidth]{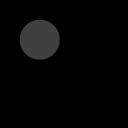}
		\caption{Unknown image.}
	\end{subfigure}	
	\begin{subfigure}[t]{0.19\textwidth}
		\centering
		\includegraphics[width = 0.99\textwidth]{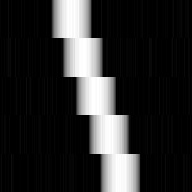}
		\caption{Measured noisy (sinogram) data.}
	\end{subfigure}	
	\begin{subfigure}[t]{0.19\textwidth}
		\centering
		\includegraphics[width = 0.99\textwidth]{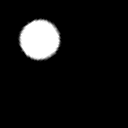}
		\caption{Transport equation, SSIM 0.880.}
		\label{fig:circlesgeo}
	\end{subfigure}
	\begin{subfigure}[t]{0.19\textwidth}
		\centering
		\includegraphics[width = 0.99\textwidth]{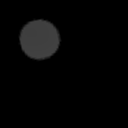}
		\caption{Continuity equation, SSIM 0.922.}
		\label{fig:circlesmp}
	\end{subfigure}
	\caption{Reconstructions of an image showing a disk obtained with our method. Five measurements were taken at directions corresponding to five angles equally distributed in $[0, 90]$ degrees. As before, five percent noise was added.}
	\label{fig:circles}	
\end{figure}

\begin{figure}[b]
	\captionsetup[subfigure]{justification=centering}
	\centering
	\begin{subfigure}[t]{0.24\textwidth}
		\centering
		\includegraphics[width = 0.98\textwidth]{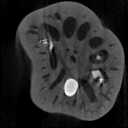}
		\caption{Template image.}
	\end{subfigure}
	\begin{subfigure}[t]{0.24\textwidth}
		\centering
		\includegraphics[width = 0.98\textwidth]{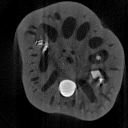}
		\caption{Unknown image.}
	\end{subfigure}
	\begin{subfigure}[t]{0.24\textwidth}
		\centering
		\includegraphics[width = 0.98\textwidth]{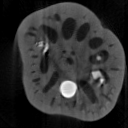}
		\caption{Reconstruction, SSIM 0.984.}
	\end{subfigure}
	\begin{subfigure}[t]{0.24\textwidth}
		\centering
		\includegraphics[width = 0.98\textwidth]{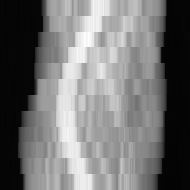}
		\caption{Measured Radon transform data.}
	\end{subfigure}
	\begin{subfigure}[t]{0.24\textwidth}
		\centering
		\includegraphics[width = 0.98\textwidth]{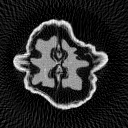}
		\caption{Template image.}
	\end{subfigure}
	\begin{subfigure}[t]{0.24\textwidth}
		\centering
		\includegraphics[width = 0.98\textwidth]{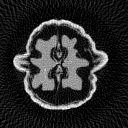}
		\caption{Unknown image.}
	\end{subfigure}
	\begin{subfigure}[t]{0.24\textwidth}
		\centering
		\includegraphics[width = 0.98\textwidth]{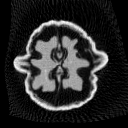}
		\caption{Reconstruction, SSIM 0.992.}
	\end{subfigure}
	\begin{subfigure}[t]{0.24\textwidth}
		\centering
		\includegraphics[width = 0.98\textwidth]{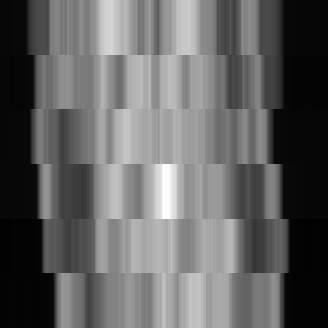}
		\caption{Measured Radon transform data.}
	\end{subfigure}
	\caption{Reconstructions based on nonsynthetic X-ray tomographic measurements \cite{BubHauHuoRimSil16, HamHarKalKujNie15} computed with our method using the transport equation together with curvature regularisation. Measurements from twelve and six directions with angles in $[0, 180]$ degrees were used.}
	\label{fig:realdata}
\end{figure}

\begin{figure}[h]
	\captionsetup[subfigure]{justification=centering}
	\centering
	\begin{subfigure}[t]{0.49\textwidth}
		\centering
		\includegraphics[width = 0.98\textwidth]{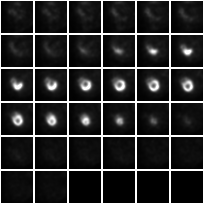}
		\caption{Template volume.}
		\label{fig:mice3D:template}
	\end{subfigure}
	\begin{subfigure}[t]{0.49\textwidth}
		\centering
		\includegraphics[width = 0.98\textwidth]{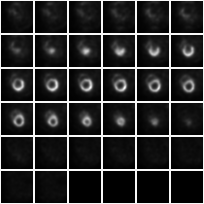}
		\caption{Unknown volume.}
		\label{fig:mice3D:unknown}
	\end{subfigure}
	\begin{subfigure}[t]{0.49\textwidth}
		\centering
		\includegraphics[width = 0.98\textwidth]{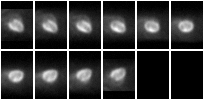}
		\caption{Measured noisy Radon transform data.}
		\label{fig:mice3D:data}
	\end{subfigure}
	\begin{subfigure}[t]{0.49\textwidth}
		\centering
		\includegraphics[width = 0.98\textwidth]{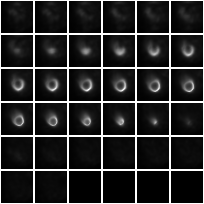}
		\caption{Reconstruction using the SSD distance, curvature regularisation, and the transport equation, SSIM 0.887.}
		\label{fig:mice3D:reconstruction}
	\end{subfigure}
	\caption{Reconstructions of a 3D volume (`mice3D', see \cite{Mod09}) using our method. In \cref{fig:mice3D:template}, \cref{fig:mice3D:unknown}, and \cref{fig:mice3D:reconstruction}, slices (left to right, top to bottom) of each volume along the third coordinate direction are shown. In \cref{fig:mice3D:data}, slices of the 3D Radon transform measurements are shown. Each slice corresponds to one measurements direction. In total, only ten measurements were taken at angles equally distributed in $[0, 180]$ degrees. As before, five percent noise was added.}
	\label{fig:mice3D}
\end{figure}

\section{Conclusions}
Overall, our numerical examples show that our implementation yields good results, as long as the deformation between template and target is fairly regular. By using the NCC-based distance, robustness with respect to intensity differences between the template and the target image can be achieved.
As already mentioned in the introduction, we do not follow the metamorphosis approach, since there is too much flexibility in the model and the source term is very likely to reproduce noise and artefacts if the data is too limited.
It is left for further research to investigate possible adaptations of the model that allow for the appearance of new objects or structures in the reconstruction without reproducing noise or artefacts.
Possibly, the results of our method can be used as better template for other algorithms that require template information.
Finally, note that due to the great flexibility of the FAIR library, it is also possible to use a great variety of regularisation functionals for the velocities and other distances, see \cite[Chaps.~7 and 8]{Mod09}.
Additionally, our implementation is not necessarily restricted to the Radon transform and essentially every (continuous) operator can be used. The multilevel approach can be applied as long as a meaningful resampling procedure for the operator and the measured data can be provided.

\section*{Acknowledgments}
Lukas F.~Lang and Carola-Bibiane Sch\"{o}nlieb acknowledge support from the Leverhulme Trust project ``Breaking the non-convexity barrier'', the EPSRC grant EP/M00483X/1, the EPSRC Centre Nr.\ EP/N014588/1, the RISE projects ChiPS and NoMADS, the Cantab Capital Institute for the Mathematics of Information, and the Alan Turing Institute.
Sebastian Neumayer is funded by the German Research Foundation (DFG) within the Research Training  Group  1932, project area P3.
Ozan \"{O}ktem is supported by the Swedish Foundation of Strategic Research, grant AM13-0049.
We gratefully acknowledge the support of NVIDIA Corporation with the donation of the Quadro P6000 GPU used for this research.

\def\cprime{$'$} \providecommand{\noopsort}[1]{}

\end{document}